\documentclass{amsart}
\usepackage[utf8]{inputenc}

\usepackage{tikz}
\usetikzlibrary{arrows.meta}
\usepackage{commutative-diagrams}
\usetikzlibrary{cd}
\usepackage{geometry}
\usepackage[colorlinks=true]{hyperref}
\usepackage[all]{hypcap}
\usepackage{mathtools}
\usepackage{booktabs}
\usepackage{tabularx}
\usepackage{array}
\usepackage{amssymb}
\usepackage{dsfont}
\usepackage{witharrows}
\usepackage{threeparttable}

\newtheorem{theorem}{Theorem}[section]
\newtheorem{lemma}[theorem]{Lemma}
\newtheorem{corollary}[theorem]{Corollary}
\newtheorem{proposition}[theorem]{Proposition}

\theoremstyle{definition}
\newtheorem{definition}[theorem]{Definition}

\theoremstyle{remark}
\newtheorem{remark}[theorem]{Remark}

\newcommand{\qbin}[2]{\genfrac{[}{]}{0pt}{}{{#1}}{{#2}}}
\newcommand{\qnum}[1]{[{#1}]}
\newcommand{\pochq}[2]{({#1};{q})_{{#2}}}

\newcommand{\hyp}[3]{\,\mbox{}_{3}\phi_{2}
	\left(\!\!\left.\begin{array}{c}#1\\#2\end{array}\right|#3\right)}

\numberwithin{equation}{section}

\newcolumntype{C}{>{\centering\arraybackslash}X}

\begin{document}

\title[]{On the dimension of the space generated by characteristic vectors of $q$-Steiner systems}


\author[Q. Li]{Qilong Li\textsuperscript{\,1}}
\address{\textsuperscript{\,1}College of Science, National University of Defense Technology, 410073 Changsha, China}
\email[Q. Li]{li.qilong@outlook.com}
\author[C. Wei{\ss}]{Charlene Wei{\ss}\textsuperscript{\,2}}
\email[C. Wei\ss]{chweiss@math.upb.de}
\author[Y. Zhou]{Yue Zhou\textsuperscript{\,1}}
\address{\textsuperscript{\,1}College of Science, National University of Defense Technology, 410073 Changsha, China}
\email[Y. Zhou]{yue.zhou.ovgu@gmail.com}



\begin{abstract}
	Fix a prime power $q$ and parameters $1\leq t\leq k\leq n$, the corresponding Steiner system in the Grassmann scheme, or the $q$-Steiner system, is a collection $\mathfrak{B}$ of $k$-dimensional subspaces of $\mathbb{F}_{q}^n$ such that for each $t$-dimensional subspace $T$, there exists exactly one element of $\mathfrak{B}$ containing $T$. The dimension of Steiner systems in the Grassmann scheme is defined to be the dimension of the $\mathbb{Q}$-vector space spanned by the characteristic vectors of all these $q$-Steiner systems. In this paper, we prove that when a quadruple $(t,k,n,q)$ admits at least one $q$-Steiner system, the corresponding dimension is equal to ${n\brack k}_{q}-{n\brack t}_{q}+1$. This generalizes the 2019 work of Ghodrati \cite{ghodrati2019dimension} on ordinary Steiner systems.
\end{abstract}

\maketitle

\section{Introduction}
Let $q>1$ be a prime power and let $\mathbb{F}_{q}$ be the finite field of order $q$. The collection of all $k$-dimensional subspaces of $\mathbb{F}_{q}^n$ is called the \emph{Grassmannian} over $\mathbb{F}_{q}$ and is denoted by $\text{Gr}_{n,k}(\mathbb{F}_{q})$. For any integer $n$ and nonnegative integer $k$,
\begin{align*}
	{n\brack k}_{q}:=\frac{(q^n-1)(q^{n-1}-1)\cdots(q^{n-k+1}-1)}{(q^k-1)(q^{k-1}-1)\cdots(q-1)}
\end{align*}
is called the \textit{Gaussian binomial coefficient} (or $q$-\textit{binomial coefficient}). It is well known that for all $0\leq k\leq n$ we have
\begin{align*}
	\#\text{Gr}_{n,k}(\mathbb{F}_{q})={n\brack k}_{q},
\end{align*}
where we define ${n\brack 0}_{q}=1$, while ${n\brack k}_{q}=0$ for all $k<0$. Let ${n\brack 1}_{q}$ be denoted by the \textit{$q$-analog integer} $[n]_{q}$.\par 
By a \textit{subspace design} with parameters $t$-$(n,k,\lambda)_{q}$ ($1\leq t\leq k\leq n$) we mean a subset of the Grassmannian $\text{Gr}_{n,k}(\mathbb{F}_{q})$ (the elements of this subset are called \emph{blocks}) such that every $t$-dimensional subspace $T$ of $\mathbb{F}_{q}^n$ is contained in exactly $\lambda$ blocks. A $t$-$(n,k,\lambda)_{q}$ design is said to be \textit{non-trivial} if $t<k$ and not every $k$-dimensional subspace serves as a block.\par 
The subspace designs were first introduced by Cameron \cite{cameron1974generalisation,cameron1974locally} and Delsarte \cite{delsarte1976association} as vector space analogs of combinatorial designs, which can be viewed as the limiting case $q\to 1$ of subspace designs \cite{tits1957analogues}. Indeed, subspace designs have many properties analogous to those of combinatorial ones. For example, a combinatorial $t$-$(n,k,\lambda)$ design actually serves as an $i$-$(n,k,\lambda_{i})$ design for any $i\in\left\{ 0,1,...,t\right\}$, with parameter
\begin{align*}
	\lambda_{i}=\dfrac{{n-i\choose t-i}}{{k-i\choose t-i}}\lambda.
\end{align*}
Similarly, the \textit{divisibility condition} holds for subspace designs in the sense of the following theorem, from which we see that $\lambda{n-i\brack t-i}_{q}$ must be divided by ${k-i\brack t-i}_{q}$ for all $0\leq i\leq t$ for a $t$-$(n,k,\lambda)_{q}$ subspace design to exist.
\begin{theorem}[{\cite[Lemma 4.1(1)]{suzuki1990inequalities}}]
	Let $t,n,k$ and $\lambda$ be positive integers satisfying $t\leq k\leq n$. Let $\mathfrak{D}$ be a $t$-$(n,k,\lambda)_{q}$ subspace design. Then for each $i\in\left\{0,1,...,t\right\}$, $\mathfrak{D}$ is an $i$-$(n,k,\lambda_{i})_{q}$ design with
	\begin{align*}
		\lambda_{i}=\lambda{n-i\brack t-i}_{q}\big/{k-i\brack t-i}_{q}.
	\end{align*}
\end{theorem}
Despite of the strong analogies and historical parallels between combinatorial and subspace designs, the situation seems to differ in essence when it comes to the latter. The existence problem of combinatorial designs has been resolved in 1987 by Teirlinck  \cite{teirlinck1987non}, who states that non-trivial combinatorial $t$-$(n,k,\lambda)$ designs exist for any $t$ and certain $\lambda$. Nevertheless, the longstanding existence problem of subspace designs had been open until resolved by Fazeli et al.\ \cite{fazeli2014nontrivial} in 2014, who used probabilistic methods rather than explicit constructions.\par 
The $q$-Steiner system defined as follows is a special case of subspace design.
\begin{definition}
	Let $t,k$ and $n$ be integers satisfying $1\leq t<k\leq n$ and let $q$ be a prime power. Then a \textit{$q$-Steiner system} $\mathfrak{B}_{q}(t,k,n)$ refers to a $t$-$(n,k,\lambda)_{q}$ subspace design with $\lambda=1$. The collection of all such $\mathfrak{B}_{q}(t,k,n)$ with fixed $t,k,n,q$ is denoted by $S_{q}(t,k,n)$. When there is no ambiguity, we write $\mathfrak{B}$ instead of $\mathfrak{B}_{q}(t,k,n)$ for brevity.
\end{definition}\par 
Note that whenever a non-trivial Steiner system with parameters $t,k,n$ exists, it is proved by Schwartz and Etzion \cite[Theorem 3]{schwartz2002codes} that
\begin{align*}
	n\geq 2k,
\end{align*}
which will be kept as an assumption in the sequel.\par 
Although the notion of the $q$-Steiner system has been studied extensively, non-trivial examples of $q$-Steiner systems have been constructed only in sporadic cases. Metsch \cite{metsch1999bose} once conjectured that $q$-Steiner system does not exist for $t\geq 2$, while this has been disproved by Braun et al.\ \cite{braun2016existence}, who explicitly constructed over 500 nonisomorphic $q$-Steiner systems in $S_{q}(2,3,13)$. Until very recently, the existence problem of $q$-Steiner system, which had lasted for five decades, was settled by Keevash et al.\ by proving the following result.
\begin{theorem}[{\cite[Theorem 1.2]{keevash2025existence}}]\label{thm_existence.of.q-Steiner.system}
	Fix $q,k$ and $t$. For $n\geq n_{1}(q,k)$ such that ${k-i\brack t-i}_{q}\mid{n-i\brack t-i}_{q}$ for all $0\leq i\leq t-1$, there is a $q$-Steiner system $\mathfrak{B}_{q}(t,k,n)$.
\end{theorem}
Furthermore, a corollary about the number of $q$-Steiner systems was established.
\begin{corollary}[{\cite[Corollary 1.3]{keevash2025existence}}]\label{coro_the.number.of.q-steiner.systems}
	Under the assumption of Theorem \ref{thm_existence.of.q-Steiner.system}, for $n\geq n_{2}(q,k)$, the number of $q$-Steiner systems in $S_{q}(t,k,n)$ equals
	\begin{align*}
		N=\left((1\pm q^{-c(t,k)\cdot n})\dfrac{{n-t\brack k-t}_{q}}{\exp({k\brack t}_{q}-1)}\right)^{{n\brack t}_{q}/{k\brack t}_{q}},
	\end{align*}
	where $c(t,k)$ is a positive constant relying on $t$ and $k$.
\end{corollary}
Therefore, throughout this paper we may assume the existence of $q$-Steiner systems, as we restrict our analysis to the quadruples $(t,k,n,q)$ that admit such systems.\par 
Historically, there have been numerous estimates on the number of combinatorial $t$-$(n,k,\lambda)$ designs from various points of view. For example, Graver and Jurkat \cite{graver1973module} initially viewed (signed) $t$-designs as integral functions, say assigning each block an integral multiplicity, thereby investigated the $\mathbb{Z}$-module structure of combinatorial designs. Graham et al.\ \cite{graham1980structure} derived explicit basis for such $\mathbb{Z}$-modules by introducing associated polynomials, and the viewpoint of associating $t$-designs with module structure turns out to be applicable for constructing new $t$-designs \cite{hwang1986structure}. Another approach is to find the dimension of the vector spaces generated by those combinatorial structures. In particular, we define \textit{the space generated by $q$-Steiner systems} as follows.
\begin{definition}
	For any $q$-Steiner system $\mathfrak{B}\in S_{q}(t,k,n)$, let the \textit{characteristic vector} $\chi_\mathfrak{B}\in\left\{ 0,1\right\}^{{n\brack k}_{q}}$ of $\mathfrak{B}$ be indexed by elements in $\text{Gr}_{n,k}(\mathbb{F}_{q})$, where the $A$-entry equals $1$ when $A$ serves as a block in $\mathfrak{B}$, and equals $0$ otherwise. Then the \textit{space generated by all $q$-Steiner systems in $S_{q}(t,k,n)$} is defined to be the $\mathbb{Q}$-vector space
	\begin{align*}
		\mathfrak{V}_{q}(t,k,n):=\text{span}_{\mathbb{Q}}\left\{\chi_{\mathfrak{B}}\mid\mathfrak{B}\in S_{q}(t,k,n)\right\}.
	\end{align*}
	Moreover, by the \textit{dimension of $q$-Steiner systems in $S_{q}(t,k,n)$} we mean $\dim_{\mathbb{Q}}\mathfrak{V}_{q}(t,k,n)$.
\end{definition}
Using such notation, Ghodrati \cite{ghodrati2019dimension} proved that for ordinary Steiner systems with parameters $t,k$ and $n$, the corresponding dimension is $\binom{n}{k}-\binom{n}{t}+1$, assuming $1\leq t\leq k\leq n/2$ and that there exists at least one such Steiner system. This has recently been generalized to designs in the bilinear forms scheme by the first and the third author of this paper \cite{li2026dimension}. In this paper, we show that the dimension of $\mathfrak{V}_{q}(t,k,n)$ is given by the following theorem.
\begin{theorem}\label{thm_main}
	Let $q$ be a prime power, let $t,k,n$ be integers satisfying $1\leq t<k\leq n$ and assume that the quadruple $(t,k,n,q)$ admits at least one $q$-Steiner system in $S_{q}(t,k,n)$. Then the dimension of $q$-Steiner systems in $S_{q}(t,k,n)$ is equal to $\dim_{\mathbb{Q}}\mathfrak{V}_{q}(t,k,n)={n\brack k}_{q}-{n\brack t}_{q}+1$.
\end{theorem}

\section{Preliminaries}
\subsection{$q$-Steiner systems and the Grassmann scheme}
From this section on, we write ${n\brack k}$ for the Gaussian binomial coefficient ${n\brack k}_{q}$ for brevity, omitting the fixed prime power $q$ unless otherwise specified.\par 
Let $U$ be the matrix whose rows and columns are indexed by elements in $\text{Gr}_{n,k}(\mathbb{F}_{q})$ and $q$-Steiner systems in $S_{q}(t,k,n)$, respectively. The $(X,\mathfrak{B})$-entry of $U$ equals $1$ if $X\in\mathfrak{B}$ and equals $0$ otherwise. In other words, the column vectors of $U$ consist of characteristic vectors $\chi_{\mathfrak{B}}$ of $\mathfrak{B}\in S_{q}(t,k,n)$, and $U$ is called the \textit{incidence matrix} of $S_{q}(t,k,n)$. To study the incidence matrix $U$, we need the Grassmann scheme which is defined as follows.
\begin{definition}
	In the Grassmannian $\text{Gr}_{n,k}(\mathbb{F}_{q})$, define relations
	\begin{align*}
		R_{i}:=\left\{(X,Y)\mid X,Y\in \text{Gr}_{n,k}(\mathbb{F}_{q}),\dim{(X\cap Y)}=k-i\right\},~i=0,1,...,k.
	\end{align*}
	Then
	\begin{align*}
		G_{q}(n,k):=(\text{Gr}_{n,k}(\mathbb{F}_{q}),\left\{ R_{i}\right\}_{0\leq i\leq k})
	\end{align*}
	forms an association scheme and is called a \textit{Grassmann scheme}.
\end{definition}
Blocks in a $q$-Steiner system $\mathfrak{B}$ are naturally distributed pairwise into the relations $R_{0},R_{1},...,R_{k}$, for which $\mathfrak{B}$ is also called a \textit{Steiner system in the Grassmann scheme}. Let $A_{i}$ be the adjacency matrices of the relations $R_{i}$, $0\leq i\leq k$. The rows and columns of each $A_{i}$ are both indexed by $\text{Gr}_{n,k}(\mathbb{F}_{q})$, where the $(X,Y)$-entry of $A_{i}$ equals $1$ when $(X,Y)\in R_i$ and equals $0$ otherwise. For distinct $X,Y\in \text{Gr}_{n,k}(\mathbb{F}_{q})$ with $\dim(X\cap Y)=i$, $0\leq i\leq k$, we define
\begin{align}\label{eq_definition.kappa}
	\kappa:=\#\left\{\mathfrak{B}\in S_{q}(t,k,n)\mid X\in\mathfrak{B}\right\}
\end{align}
and
\begin{align}\label{eq_definition.kappa.i}
	\kappa_{i}:=\#\left\{\mathfrak{B}\in S_{q}(t,k,n)\mid X\in\mathfrak{B},Y\in\mathfrak{B}\right\},~i=0,1,...,k.
\end{align}
Suppose $\kappa$ is defined for an element $X'\in \text{Gr}_{n,k}(\mathbb{F}_{q})$ distinct from $X$, then there exists an automorphism of $\mathbb{F}_{q}^n$ sending $X'$ to $X$, under which the Steiner systems in $G_{q}(n,k)$ having $X'$ as a block become the ones having $X$ as a block. Hence $\kappa$ is independent from the choice of $X$. The same argument goes for $\kappa_{i}$ by considering the automorphism of $\mathbb{F}_{q}^n$ that sends $X\cap Y$ to $X'\cap Y'$, so $\kappa_{i}$ depends only on $\dim(X\cap Y)=i$. Note that for each $i$-dimensional subspace $I$, since $\mathfrak{B}$ is a Steiner system, there exist $X,Y\in\mathfrak{B}$ such that $X\neq Y$ and $X\cap Y=I$ only if $0\leq i\leq t-1$. This implies
\begin{align*}
	\text{$\kappa_{i}=0$~~for all~~$t\leq i\leq k$}.
\end{align*}
Therefore, we obtain
\begin{align}\label{eq_gram.matrix.decomposition}
	UU^{T}=\kappa I+\sum_{i=0}^{t}\kappa_{i}A_{k-i}.
\end{align}\par 
If we assume the conditions of Corollary \ref{coro_the.number.of.q-steiner.systems}, which ensure that $S_{q}(t,k,n)$ is nonempty, then the dimension of $q$-Steiner systems in $S_{q}(t,k,n)$ is exactly $\text{rank}\left(U\right)$. This yields
\begin{align*}
	\begin{split}
		\dim_{\mathbb{Q}}\mathfrak{V}=\text{rank}\left(U\right)=\text{rank}\left(UU^{T}\right)&={n\brack k}-\text{null}\left(UU^{T}\right)\\
		&={n\brack k}-\text{multiplicity of zero as an eigenvalue of $UU^{T}$},
	\end{split}
\end{align*}
where $\text{null}(\cdot)$ refers to the nullity of a matrix. Hence, to determine the dimension of $q$-Steiner systems in $S_{q}(t,k,n)$ it suffices to study the eigenvalues of the Gram matrix $UU^{T}$, which we will also refer to as the \textit{eigenvalues of $q$-Steiner systems}.
\subsection{Some well-known combinatorial identities}
To compute the eigenvalues of $UU^{T}$, we need the \emph{$q$-Pochhammer symbol}~$\pochq{a}{n}$ defined by
\[
\pochq{a}{0}=1,\quad \pochq{a}{n}=\prod_{i=0}^{n-1}(1-aq^i)
\]
for a positive integer $n$ and a real number $a$. The $q$-Pochhammer symbol is connected to the Gaussian binomial coefficient by 
\begin{align}\label{eq:qbin_only_poch}
	\qbin{n}{k}=\frac{\pochq{q}{n}}{\pochq{q}{k}\pochq{q}{n-k}}.
\end{align}
We will also use the following well-known identities:
\begin{align}\label{eq:qbin_only_poch_sum}
	\qbin{n+k}{n}=\frac{\pochq{q^{k+1}}{n}}{\pochq{q}{n}}
\end{align}
\begin{align}\label{eq:qbin_poch}
	\qbin{n}{k}=&\,\frac{\pochq{q^{-n}}{k}}{\pochq{q}{k}} (-1)^k q^{kn-\binom{k}{2}},
\end{align}
which together with \eqref{eq:qbin_only_poch} can be found in \cite[\S~1.9]{koekoek2010hypergeometric}. It is also well known that
\begin{align}\label{eq:poch_diff}
	\pochq{q}{n-k}=\frac{\pochq{q}{n}}{\pochq{q^{-n}}{k}} (-1)^k q^{\binom{k}{2}-nk}
\end{align}
\begin{align}\label{eq:poch_sum}
	\pochq{q}{n+k}=\pochq{q}{n}\pochq{q^{n+1}}{k},
\end{align}
see \cite[Eq.~(1.8.10), (1.8.16)]{koekoek2010hypergeometric}. The Gaussian binomial coefficient satisfies the upper negation formula
\begin{align}\label{eq:uppernegation}
	\qbin{n}{k}=(-1)^k q^{kn-\binom{k}{2}}\qbin{k-n-1}{k},
\end{align}
as well as the $q$-binomial theorem
\begin{align}\label{eq:qbinomialthm}
	\sum_{k=0}^n \qbin{n}{k} q^{\binom{k}{2}} x^k y^{n-k}=\prod_{i=0}^{n-1} (xq^i+y),
\end{align}
see \cite[Ex.~1.2(iv)]{gasper2004basic} and \cite[\S~3, Ex.~45.a]{stanley2012enumerative}, respectively.
We also need the \emph{$q$-hypergeometric function}~$\mbox{}_{3}\phi_{2}$ defined by
\[
\hyp{a_{1},\dots,a_{r}}{b_{1},\dots,b_s}{q,z}=\sum_{\ell=0}^{\infty} \frac{(a_{1};q)_\ell\cdots(a_{r};q)_\ell}{(b_{1};q)_\ell\cdots(b_s;q)_\ell}\,\frac{z^\ell}{(q;q)_\ell}.
\]
Moreover, we will use the transformation formula
\begin{align}\label{eq:transf32}
	\hyp{q^{-n}, a, b}{c, d}{q;q}=\frac{\pochq{a^{-1}b^{-1}cd}{n}}{\pochq{d}{n}} \left(\frac{ab}{c}\right)^n \hyp{q^{-n}, a^{-1}c, b^{-1}c}{c, a^{-1}b^{-1}cd}{q;q},
\end{align}
which can be found in \cite[Eq.~(III.11)]{gasper2004basic}. The following identities will also be used.
\begin{proposition}[{\cite[Proposition A.3(4)]{kurihara2013character}}]\label{prop_kurihara.triterms}
	For any integers $x,y$ and nonnegative integers $h,p$,
	\begin{align*}
		\begin{split}
			{x\brack h}{y-x\brack p-h}&=\sum_{v=h}^{p}(-1)^{v-h}{v\brack h}{y-v\brack p-v}{x\brack v}q^{-(p-h)(x-h)+{v-h\choose 2}}\\
			&=\sum_{v=h}^{p}(-1)^{v-h}{v\brack h}{y-v\brack p-v}{x\brack v}q^{(v-h)(y-x-p+h)+{v-h+1\choose 2}}.
		\end{split}
	\end{align*}
\end{proposition}
\begin{proposition}[{\cite[Lemma 3]{lv2012eigenvalues}}]\label{prop_lv.wang.binomial}
	For any integer $x$ and nonnegative integer $a$,
	\begin{align*}
		\sum_{v=0}^{a}(-1)^v{x\brack v}q^{{v\choose 2}}=q^{xa}{a-x\brack a}.
	\end{align*}
\end{proposition}
\subsection{M{\"o}bius inversion}
We recall some basics of M{\"o}bius inversion and we refer to \cite[\S~1.2]{hou2018lectures} for a comprehensive introduction.
\begin{definition}
	A \textit{partially ordered set} (or a \textit{poset} for short) refers to a nonempty set $X$ with binary relation $\leq$ satisfying\par 
	(i) $x\leq x$ for all $x\in X$;\par 
	(ii) if $x\leq y$ and $y\leq z$ for $x,y,z\in X$, then $x\leq z$;\par
	(iii) if $x\leq y$ and $y\leq x$ for $x,y\in X$, then $x=y$.\par 
	The relation $\leq$ is called a \textit{partial order} on $X$. Furthermore, if for any $x,y\in X$, the interval $[x,y]:=\left\{ z\in X\mid x\leq z, z\leq y\right\}$ has finite cardinality, then the poset $(X,\leq)$ is said to be \textit{locally finite}.
\end{definition}
\begin{definition}
	Let $(X,\leq)$ be a locally finite poset. The \textit{M{\"o}bius function} on $X$ is defined to be a function
	\begin{align*}
		\mu:X\times X\to\mathbb{Z}
	\end{align*}
	satisfying $\mu(x,y)=0$ whenever $x\not\leq y$, and when $x\leq y$ one has
	\begin{align*}
		\sum_{z\in[x,y]}\mu(x,z)=\delta(x,y),
	\end{align*}
	where $\delta(x,y)=\begin{cases}
		1, & x=y\\
		0, & x\neq y
	\end{cases}$ denotes the \textit{Kronecker symbol}.
\end{definition}
It can be deduced from the definition that the M{\"o}bius function of a locally finite poset exists and is unique. We need the following celebrated formula of the so-called \textit{M{\"o}bius inversion}.
\begin{lemma}\label{lemma_mobius.inversion}
	Let $(X,\leq)$ be a locally finite poset with M{\"o}bius function $\mu$. Let $A$ be an Abelian group and let $\mathcal{N}:X\to A$ be a function. Fix $x_{0}\in X$, and for $x\in X$ define
	\begin{align*}
		\mathcal{N}_{\leq}(x):=\sum_{y\in[x_{0},x]}\mathcal{N}(y).
	\end{align*}
	Then
	\begin{align*}
		\mathcal{N}(x)=\sum_{y\in[x_{0},x]}\mu(y,x)\mathcal{N}_{\leq}(y)
	\end{align*}
	for all $x\in X$ with $x\geq x_{0}$.
\end{lemma}
When it comes to the \textit{subspace lattice} $(\mathcal{L}(V),\subseteq)$, where $\mathcal{L}(V)$ denotes the set of all subspaces of an $n$-dimensional vector space $V$ over $\mathbb{F}_{q}$, the M{\"o}bius function on $(\mathcal{L}(V),\subseteq)$ is given by
\begin{align*}
	\mu(U,W):=(-1)^{\dim{W}-\dim{U}}q^{{\dim{W}-\dim{U}\choose 2}}
\end{align*}
for all subspaces $U,W$ of $V$ with $U\subseteq W$ (cf. \cite{bender1975applications} and \cite{rota1964foundations}).

\section{Eigenvalues of $q$-Steiner systems}
Recall that for proving Theorem \ref{thm_main}, we need to investigate the eigenvalues and eigenspaces of the adjacency matrices $A_{k-t},A_{k-t+1},...,A_{k}$ of the Grassmann scheme. Let $\mathcal{L}_{k}$ denote the set of $k$-dimensional subspaces of the projective space $PG(n,q)$ (resp., the corresponding Grassmann space). In particular, define $\mathcal{L}_{-1}=\left\{\emptyset\right\}$ and $\mathcal{L}_{-2}=\emptyset$. Identify the vector $(f_{1},...,f_{{n\brack k}})\in\mathbb{R}^{{n\brack k}}$ with the function $f:\mathcal{L}_{{n\brack k}}\to\mathbb{R},X_j\mapsto f_j$. We have the following theorem which determines the eigenspaces of each adjacency matrix $A_{i}$ ($0\leq i\leq k$).
\begin{theorem}[{\cite[Theorem 2.7]{eisfeld1999eigenspaces}}]\label{thm_eigenspace.of.grassmann.scheme}
	For $r\in\left\{0,1,...,\min(k+1,n-k)\right\}$, let $V_{r}$ be the $\mathbb{R}$-vector space spanned by  functions of the form
	\begin{align*}
		\begin{split}
			f:\mathcal{L}_{k}&\to\mathbb{R}\\L_{k}&\mapsto\sum_{L_{r-1}\subseteq L_{k}}g(L_{r-1}),
		\end{split}
	\end{align*}
	where $g:\mathcal{L}_{r-1}\to\mathbb{R}$ satisfies that $\sum_{L_{r-1}\supseteq L_{r-2}}g(L_{r-1})=0$ for all $L_{r-2}\in\mathcal{L}_{r-2}$. Then $V_{r}$ is the $r$-th eigenspace of $A_{i}$ to the eigenvalue
	\begin{align*}
		\nu_{r}^{(i)}:=\sum_{j=\max(0,r-i)}^{\min(r,k-i)}(-1)^{r-j}{r\brack j}{n-k+j-r\brack n-k-i}{k-j\brack i}q^{i(i+j-r)+\binom{r-j}{2}}.
	\end{align*}
	Furthermore, the spaces $V_{r}$, $r=0,1,...,\min(k,n-k)$ form the complete system of the eigenspaces of Grassmann scheme $G_{q}(n,k)$.
\end{theorem}
It can be proved that the multiplicity of $\nu_{r}^{(i)}$ equals ${n\brack r}-{n\brack r-1}$. Actually, in the proof of Theorem \ref{thm_eigenspace.of.grassmann.scheme} by Eisfeld \cite{eisfeld1999eigenspaces}, it has been shown that $\overline{V}_{r}=V_{0}\oplus V_{1}\oplus\cdots\oplus V_{r}$, where $\overline{V}_{0}=V_{0}$ and $\overline{V}_{r}$ ($r\geq 1$) is the vector space spanned by the functions $f:\mathcal{L}_{k}\to\mathbb{R}$ of the form $f(L_{k})=\sum_{L_{r-1}\subseteq L_n}g(L_{r-1})$, $g:\mathcal{L}_{r-1}\to\mathbb{R}$. Hence the assertion follows directly by induction on $r$, given that $\dim{\overline{V}_{r}}={n\brack r}$.\par 
Notice that the eigenspace $V_{r}$ is independent from the choice of $i$, which yields that the adjacency matrices $A_{i}$ have the $r$-th eigenspace $V_{r}$ in common. Therefore, together with \eqref{eq_gram.matrix.decomposition} we see that the $r$-th eigenvalue of the Gram matrix $UU^{T}$ is exactly
\begin{align*}
	\mu_{r}=\kappa+\sum_{i=0}^{t} \kappa_{i}\nu_{r}^{(k-i)},~0\leq r\leq k,
\end{align*}
where each $\nu_{r}^{(k-i)}$ has multiplicity ${n\brack r}-{n\brack r-1}$.\par 
Furthermore, there is an equivalent expression of the $r$-th eigenvalue $\nu_{r}^{(i)}$ which was given by Delsarte \cite[Theorem 10]{delsarte1976association}, namely
\begin{align}\label{eq_eigenvalue.grassmann.scheme.delsarte}
	\nu_{r}^{(i)}=E_{i}(n,k;q;r)
\end{align}
with the so-called \textit{generalized Eberlein polynomial} (cf. \cite[Sect. 5.2]{delsarte1976properties})
\begin{align*}
	E_{i}(n,k;q;x):=\sum_{j=0}^{i} (-1)^j \qbin{x}{j}\qbin{k-x}{i-j}\qbin{n-k-x}{i-j} q^{\binom{j}{2}+(i-j)(i-j+x)}.
\end{align*}
Henceforth, we will adopt the expression \eqref{eq_eigenvalue.grassmann.scheme.delsarte} for the eigenvalues $\nu_{r}^{(i)}$ for the convenience of computation.
\subsection{Determining the coefficients $\kappa$ and $\kappa_{i}$}
Recall from \eqref{eq_definition.kappa} and \eqref{eq_definition.kappa.i} the definitions of $\kappa$ and $\kappa_{i}$, respectively. Also recall that the number of $q$-Steiner systems in $S_{q}(t,k,n)$, denoted by $N$, is given by Corollary \ref{coro_the.number.of.q-steiner.systems} and is non-zero.

\begin{proposition}\label{prop_kappa}
	We have
	\begin{align*}
		\kappa=N\dfrac{{n\brack t}}{{n\brack k}{k\brack t}}.
	\end{align*}
	\begin{proof}
		Consider the set $\Omega:=\left\{(\mathfrak{B},X)\mid \mathfrak{B}\in S_{q}(t,k,n),X\in\mathfrak{B}\right\}$. We obtain the value of $\kappa$ by counting $\#\Omega$ in two different ways. On the one hand, there are $N$ choices of Steiner system $\mathfrak{B}$ and there are exactly ${n\brack t}/{k\brack t}$ blocks in $\mathfrak{B}$, which is the number of choices of $X$. Hence $\#\Omega=N\cdot{n\brack t}/{k\brack t}$. On the other hand, since the number of $k$-dimensional subspaces of $V$ equals $n\brack k$, we have $\#\Omega={n\brack k}\cdot \kappa$. Thus
		\begin{align*}
			N\cdot{n\brack t}\big/{k\brack t}=\#\Omega={n\brack k}\cdot \kappa,
		\end{align*}
		which gives the result.
	\end{proof}
\end{proposition}
For computing $\kappa_{i}$, $i=0,1,...,t$, some auxiliary lemmas are needed.
\begin{lemma}\label{lemma_spanning.sets.counting}
	Let $m$ be a positive integer and $d$ be a nonnegative integer. For a $d$-dimensional subspace $V_{d}$ of a vector space $V$ over $\mathbb{F}_{q}$ we define
	\begin{align*}
		N_{q}(m,d)=N_{q}(m,V_{d}):=\#\left\{v_{1},...,v_{m}\in\text{PG}(V)\mid\text{span}\left\{v_{1},...,v_{m}\right\}=V_{d}\right\},
	\end{align*}
	where vectors $v_{1},...,v_{m}\in V$ lying in the projective space $\text{PG}(V)$, by abuse of language, means that $v_{1},...,v_{m}$ are nonzero and pairwise non-collinear. Then $N_{q}(m,d)=0$ whenever $1\leq m<d$, and $N_{q}(m,0)=0$ for any $m\geq 1$. In general,
	\begin{align}\label{eq_spanning.sets.counting.formula}
		N_{q}(m,d)=\sum_{j=0}^{d}{d\brack j}(-1)^{d-j}q^{{d-j\choose 2}}{[j]\choose m}.
	\end{align}
	\begin{proof}
		By the definition of $N_{q}(m,d)$, we have
		\begin{align*}
			\sum_{0\subseteq V'\subseteq V_d}N_{q}(m,V')={[d]\choose m}
		\end{align*}
		for any $m,d$ and $d$-dimensional subspace $V_{d}$, since by summing up the number of spanning sets with $m$ elements for all subspaces of $V_{d}$, we are choosing all subsets of $m$ pairwise non-collinear nonzero vectors in $V_{d}$. Hence by applying Lemma \ref{lemma_mobius.inversion}, one obtains
		\begin{align*}
			\begin{split}
				N_{q}(m,d)=&\sum_{0\subseteq V'\subseteq V_{d}}\mu(V',V_d)\sum_{0\subseteq V''\subseteq V'}N_{q}(m,V'')\\
				=&\sum_{j=0}^{d}{d\brack j}(-1)^{d-j}q^{{d-j\choose 2}}{[j]\choose m}.\qedhere
			\end{split}
		\end{align*}
	\end{proof}
\end{lemma}
\begin{remark}
	In \cite{bender1975applications}, it was shown that
	\begin{align*}
		\sum_{j=0}^{d}{d\brack j}(-1)^{d-j}q^{{d-j\choose 2}}(2^{[j]}-1)
	\end{align*}
	actually counts the number of spanning sets for projective spaces, which is precisely the sum \eqref{eq_spanning.sets.counting.formula} by letting $m$ run from $1$ to $[d]$.
\end{remark}
We also note the following well-known result with regards to the subspace counting.
\begin{lemma}[{\cite[Lemma 1]{braun2018q}}]\label{lemma_counting.subspace}
	Let $a,b,u\in\left\{ 0,1,...,n\right\}$ and let $B\in{V\brack b}$, where $V=\mathbb{F}_{q}^n$. For $A\in{V\brack a}$, we have
	\begin{align*}
		\#\left\{U\in{V\brack u}\mid U\cap B=A\right\}=q^{(b-a)(u-a)}{n-b\brack u-a}
	\end{align*}
	and
	\begin{align*}
		\#\left\{U\in{V\brack u}\mid \dim(U\cap B)=a\right\}=q^{(b-a)(u-a)}{b\brack a}{n-b\brack u-a}.
	\end{align*}
\end{lemma}
\begin{lemma}\label{lemma_intersect.counting}
	For any $\mathfrak{B}\in S_{q}(t,k,n)$, let $X$ be a block in $\mathfrak{B}$ and $I\in{X\brack i}$, where $i\in\left\{0,1,...,t-1\right\}$. Then
	\begin{multline*}
		\#\left\{ Y\in\mathfrak{B}\mid Y\neq X,X\cap Y=I\right\}=\\\frac{1}{{n-t\brack k-t}}\sum_{j=i}^{t}{k-i\brack j-i}{n-j\brack k-j}(-1)^{j-i}q^{{j-i\choose 2}}+(-1)^{t+1-i}{k-i-1\brack t-i}q^{{t+1-i\choose 2}}.
	\end{multline*}
	\begin{proof}
		Let $\hat{I}$ be an $(i+1)$-dimensional subspace of $X$ containing $I$, then it is clear that there are precisely $[k-i]$ choices of $\hat{I}$. Let $\hat{I}_{1},...,\hat{I}_{[k-i]}$ be all such spaces. It follows that
		\begin{align}\label{eqs_inclusion.exclusion}
			\begin{split}
				&\#\left\{Y\in\mathfrak{B}\mid Y\neq X,X\cap Y=I\right\}\\
				=&\#\left(\left\{ Y\in\mathfrak{B}\mid Y\neq X,X\cap Y\supseteq I\right\}\big\backslash\bigcup_{l=1}^{[k-i]}\left\{ Y\in\mathfrak{B}\mid Y\neq X,X\cap Y\supseteq\hat{I}_l\right\}\right)\\
				=&\#\left\{ Y\in\mathfrak{B}\mid Y\neq X,X\cap Y\supseteq I\right\}\\
				&\qquad\quad-\sum_{l=1}^{[k-i]}(-1)^{l+1}\sum_{1\leq i_{1}<\cdots<i_l\leq[k-i]}\#\bigcap_{j=1}^l\left\{ Y\in\mathfrak{B}\mid Y\neq X,X\cap Y\supseteq\hat{I}_{i_j}\right\}\\
				=&(\lambda_{i}-1)-\sum_{l=1}^{[k-i]}(-1)^{l+1}\sum_{1\leq i_{1}<\cdots<i_l\leq[k-i]}\#\left\{ Y\in\mathfrak{B}\mid Y\neq X,X\cap Y\supseteq\sum_{j=1}^l\hat{I}_{i_j}\right\}.
			\end{split}
		\end{align}
		Let $\overline{v_{1}},...,\overline{v_{[k-i]}}$ be the generators of all $1$-dimensional subspaces of $X/I$, where
		\begin{align*}
			\overline{\left(\cdot\right)}:X\to X/I
		\end{align*}
		denotes the canonical map. Then the space $\sum_{j=1}^{l}\hat{I}_{i_{j}}$ defined for each $1\leq l\leq[k-i]$ is precisely $\text{span}_{\mathbb{F}_{q}}\left\{w,v_{i_{j}}\mid w\in I, 1\leq j\leq l\right\}$, and choosing the spanning sets of $\sum_{j=1}^{l}\hat{I}_{i_j}$ is therefore equivalent to choosing the spanning sets $\left\{\overline{v_{a}}\mid a\in S\right\}$ of $\sum_{j=1}^{l}\hat{I}_{i_j}/I$ with $\emptyset\neq S\subseteq\left\{1,2,...,[k-i]\right\}$, that is to select vectors from $\overline{v_{1}},...,\overline{v_{[k-i]}}$.\par 
		Note that once $\dim\sum_{j=1}^{l}\hat{I}_{i_{j}}/I=d$, we have
		\begin{align*}
			\#\left\{ Y\in\mathfrak{B}\mid Y\neq X,X\cap Y\supseteq\sum_{j=1}^{l}\hat{I}_{i_{j}}\right\}=\lambda_{d+i}-1,
		\end{align*}
		where $\lambda_{d+i}={n-d-i\brack t-d-i}/{k-d-i\brack t-d-i}$ with $1\leq d\leq t-i$, while
		\begin{align*}
			\left\{ Y\in\mathfrak{B}\mid Y\neq X,X\cap Y\supseteq\sum_{j=1}^{l}\hat{I}_{i_{j}}\right\}=\emptyset
		\end{align*}
		for $d\geq t-i$ since $\mathfrak{B}\in S_{q}(t,k,n)$ is a $q$-Steiner system, in which case we define $\lambda_{d+i}=1$. Therefore
		\begin{align}\label{eq_subset.counting}
			\begin{split}
				&\sum_{l=1}^{[k-i]}(-1)^{l}\sum_{1\leq i_{1}<\cdots<i_l\leq[k-i]}\#\left\{ Y\in\mathfrak{B}\mid Y\neq X,X\cap Y\supseteq\sum_{j=1}^{l}\hat{I}_{i_{j}}\right\}\\
				=&\sum_{\emptyset\neq S\subseteq\left\{ 1,2,...,[k-i]\right\}}(-1)^{|S|}(\lambda_{d_{|S|}+i}-1),
			\end{split}
		\end{align}
		where $d_{|S|}=\dim\text{span}_{\mathbb{F}_{q}}\left\{v_{a}\mid a\in S\right\}$. If we rearrange the right-hand side of \eqref{eq_subset.counting} by the dimension of subspaces that the vectors $\overline{v_{a}}$ ($a\in S$) span, together with \eqref{eqs_inclusion.exclusion} we obtain
		\begin{align*}
			\begin{split}
				&\#\left\{Y\in\mathfrak{B}\mid Y\neq X,X\cap Y=I\right\}\\
				=&(\lambda_{i}-1)-\sum_{l=1}^{[k-i]}(-1)^{l+1}\sum_{1\leq i_{1}<\cdots<i_l\leq[k-i]}\#\left\{Y\in\mathfrak{B}\mid Y\neq X,X\cap Y\supseteq\sum_{j=1}^l\hat{I}_{i_j}\right\}\\
				=&(\lambda_{i}-1)+\sum_{\emptyset\neq S\subseteq\left\{ 1,2,...,[k-i]\right\}}(-1)^{|S|}(\lambda_{d_{|S|}+i}-1)\\
				=&(\lambda_{i}-1)+\sum_{d=1}^{k-i}{k-i\brack d}(\lambda_{d+i}-1)\sum_{m=1}^{[d]}(-1)^mN_{q}(m,d)\\
				=&(\lambda_{i}-1)+\sum_{d=1}^{t-i}{k-i\brack d}(\lambda_{d+i}-1)\sum_{m=1}^{[d]}(-1)^mN_{q}(m,d),
			\end{split}
		\end{align*}
		where $N_{q}(m,d)=N_{q}(m,\overline{V_{d}})=\#\left\{\overline{v_{1}},...,\overline{v_m}\in\text{PG}(X/I)\mid\text{span}\left\{\overline{v_{1}},...,\overline{v_{m}}\right\}=\overline{V_{d}}\right\}$, $\overline{V_{d}}\in{X/I\brack d}$. Thus, by applying Lemma \ref{lemma_spanning.sets.counting} we have
		\begin{align}\label{eqs_binomial.reduction}
			\begin{split}
				&\#\left\{Y\in\mathfrak{B}\mid Y\neq X,X\cap Y=I\right\}\\
				=&(\lambda_{i}-1)+\sum_{d=1}^{t-i}{k-i\brack d}(\lambda_{d+i}-1)\sum_{m=1}^{[d]}(-1)^m\sum_{j=0}^{d}{d\brack j}(-1)^{d-j}q^{{d-j\choose 2}}{[j]\choose m}\\
				=&(\lambda_{i}-1)+\sum_{d=1}^{t-i}{k-i\brack d}(\lambda_{d+i}-1)\sum_{j=1}^{d}{d\brack j}(-1)^{d-j}q^{{d-j\choose 2}}\sum_{m=1}^{[d]}(-1)^{m}{[j]\choose m}\\
				=&(\lambda_{i}-1)-\sum_{d=1}^{t-i}{k-i\brack d}(\lambda_{d+i}-1)\sum_{j=1}^{d}{d\brack j}(-1)^{d-j}q^{{d-j\choose 2}}.
			\end{split}
		\end{align}
		Notice that for each $0\leq d\leq t-i$,
		\begin{align*}
			\lambda_{d+i}={n-d-i\brack t-d-i}\big/{k-d-i\brack t-d-i}={n-d-i\brack k-d-i}\big/{n-t\brack k-t}.
		\end{align*}
		This gives rise to
		\begin{align}\label{eq_counting.reduction.lambda}
			\begin{split}
				\sum_{d=1}^{t-i}{k-i\brack d}\lambda_{d+i}\sum_{j=1}^{d}{d\brack j}(-1)^{d-j}q^{{d-j\choose 2}}&=\sum_{d=1}^{t-i}{k-i\brack d}\lambda_{d+i}\left(-(-1)^{d}q^{{d\choose 2}}\right)\\
				&=\frac{1}{{n-t\brack k-t}}\sum_{d=1}^{t-i}{k-i\brack d}{n-d-i\brack k-d-i}(-1)^{d+1}q^{{d\choose 2}},
			\end{split}
		\end{align}
		where the first equality follows from Proposition \ref{prop_kurihara.triterms} by taking $(x,y,h,p)=(d,d,0,d)$. Similarly, we have
		\begin{align}\label{eqs_counting.reduction.1}
			\begin{split}
				\sum_{d=1}^{t-i}{k-i\brack d}\sum_{j=1}^{d}{d\brack j}(-1)^{d-j}q^{{d-j\choose 2}}=&\sum_{d=1}^{t-i}{k-i\brack d}(-1)^{d+1}q^{{d\choose 2}}\\
				=&-{t-k\brack t-i}q^{(t-i)(k-i)}+1\\
				=&-(-1)^{t-i}{k-i-1\brack t-i}q^{(t-k)(t-i)-{t-i\choose 2}+(t-i)(k-i)}+1\\
				=&(-1)^{t+1-i}{k-i-1\brack t-i}q^{{t+1-i\choose 2}}+1.
			\end{split}
		\end{align}
		Thus, by substituting \eqref{eq_counting.reduction.lambda} and \eqref{eqs_counting.reduction.1} into \eqref{eqs_binomial.reduction}, we obtain
		\begin{align*}
			\begin{split}
				&\#\left\{Y\in\mathfrak{B}\mid Y\neq X,X\cap Y=I\right\}\\
				=&(\lambda_{i}-1)-\frac{1}{{n-t\brack k-t}}\sum_{d=1}^{t-i}{k-i\brack d}{n-d-i\brack k-d-i}(-1)^{d+1}q^{{d\choose 2}}+(-1)^{t+1-i}{k-i-1\brack t-i}q^{{t+1-i\choose 2}}+1\\
				=&\frac{1}{{n-t\brack k-t}}\sum_{d=0}^{t-i}{k-i\brack d}{n-d-i\brack k-d-i}(-1)^{d}q^{{d\choose 2}}+(-1)^{t+1-i}{k-i-1\brack t-i}q^{{t+1-i\choose 2}}\\
				=&\frac{1}{{n-t\brack k-t}}\sum_{j=i}^{t}{k-i\brack j-i}{n-j\brack k-j}(-1)^{j-i}q^{{j-i\choose 2}}+(-1)^{t+1-i}{k-i-1\brack t-i}q^{{t+1-i\choose 2}}.\qedhere
			\end{split}
		\end{align*}
	\end{proof}
\end{lemma}
Now we are ready to compute $\kappa_{i}=\#\left\{\mathfrak{B}\in S_{q}(t,k,n)\mid X\in\mathfrak{B},Y\in\mathfrak{B}\right\}$ for distinct $X,Y\in\text{Gr}_{n,k}(\mathbb{F}_{q})$ with $\dim(X\cap Y)=i$.
\begin{proposition}\label{prop_kappa.i}
	We have
	\begin{align*}
		\kappa_{i}=\frac{N}{{n-t\brack k-t}{n-k\brack k-i}q^{(k-i)^2}}\left(\frac{1}{{n-t\brack k-t}}\sum_{j=i}^{t}{k-i\brack j-i}{n-j\brack k-j}(-1)^{j-i}q^{{j-i\choose 2}}+(-1)^{t+1-i}{k-i-1\brack t-i}q^{{t+1-i\choose 2}}\right).
	\end{align*}
	\begin{proof}
		Consider the set $\Omega_{i}:=\left\{(\mathfrak{B},X,Y)\mid \mathfrak{B}\in S_{q}(t,k,n),X\neq Y,X\in\mathfrak{B},Y\in\mathfrak{B},\dim(X\cap Y)=i\right\}$. We count the cardinality of this set in two different ways to obtain $\kappa_{i}$ for any $i\in\left\{0,1,...,t-1\right\}$.\par 
		On the one hand, we first choose the $i$-dimensional subspace $I$ (actually $X\cap Y$), for which we have $n\brack i$ choices. Then by Lemma \ref{lemma_counting.subspace}, the number of $X\in \text{Gr}_{n,k}(\mathbb{F}_{q})$ containing $I$ equals ${n-i\brack k-i}$, and the number of $Y\in \text{Gr}_{n,k}(\mathbb{F}_{q})$ satisfying $Y\cap X=I$ is precisely ${n-k\brack k-i}q^{(k-i)^2}$. Hence
		\begin{align}\label{eq_kappa.i.counting.1}
			\#\Omega_{i}={n\brack i}{n-i\brack k-i}{n-k\brack k-i}q^{(k-i)^2}\cdot\kappa_{i}.
		\end{align}\par 
		On the other hand, first note there are $N$ choices of a Steiner system $\mathfrak{B}$ and there are ${n\brack t}/{k\brack t}$ choices of $X$ as a block in $\mathfrak{B}$. We may thereby choose $I$ (again, actually $X\cap Y$) as an $i$-dimensional subspace of $X$ provided with $k\brack i$ choices. Thus by Lemma \ref{lemma_intersect.counting},
		\begin{align}\label{eq_kappa.i.counting.2}
			\#\Omega_{i}=N\cdot\frac{{n\brack t}{k\brack i}}{{k\brack t}}\cdot\left(\frac{1}{{n-t\brack k-t}}\sum_{j=i}^{t}{k-i\brack j-i}{n-j\brack k-j}(-1)^{j-i}q^{{j-i\choose 2}}+(-1)^{t+1-i}{k-i-1\brack t-i}q^{{t+1-i\choose 2}}\right).
		\end{align}
		Comparing \eqref{eq_kappa.i.counting.1} with \eqref{eq_kappa.i.counting.2}, we have
		\begin{multline*}
			\kappa_{i}=N\cdot\frac{{n\brack t}{k\brack i}}{{k\brack t}{n\brack i}{n-i\brack k-i}{n-k\brack k-i}q^{(k-i)^2}}\cdot\Bigg(\frac{1}{{n-t\brack k-t}}\sum_{j=i}^{t}{k-i\brack j-i}{n-j\brack k-j}(-1)^{j-i}q^{{j-i\choose 2}}\\+(-1)^{t+1-i}{k-i-1\brack t-i}q^{{t+1-i\choose 2}}\Bigg).
		\end{multline*}
		Note that ${n\brack t}{n-t\brack k-t}={k\brack t}{n\brack k}$ and ${k\brack i}{n\brack k}={n\brack i}{n-i\brack k-i}$, from which the result follows immediately.
	\end{proof}
\end{proposition}
\subsection{Derivation of the eigenvalues}
We first derive some identities required to compute the eigenvalues of $UU^{T}$.
\begin{lemma}\label{lem:identity1}
	For integers $n,r,k,u,i$ with $r\leq n+1$, we have
	\begin{align*}
		\sum_{s=0}^u (-1)^s q^{\binom{u-s}{2}}\frac{\qbin{n-r+1}{u-s}\qbin{k-i+s}{s}^2}{\qbin{r-i+s}{s}}
		=
		q^{u(2k-i-r+1)}\sum_{s=0}^u (-1)^s q^{\binom{u-s}{2}+s(2r-2k+s-1)}\frac{\qbin{n-2k+i}{u-s}\qbin{k-r}{s}^2}{\qbin{r-i+s}{s}}.
	\end{align*}
	\begin{proof}
		Let $x=n-r+1$ and $L$ denote the left-hand side of the required identity. By using \eqref{eq:qbin_only_poch}, \eqref{eq:poch_diff}, and \eqref{eq:poch_sum}, we have
		\begin{align}\label{eq:qbin_xus}
			\qbin{x}{u-s}=(-1)^s q^{us-\binom{s}{2}}\frac{\pochq{q}{x}\pochq{q^{-u}}{s}}{\pochq{q}{u}\pochq{q}{x-u}\pochq{q^{x-u+1}}{s}}.
		\end{align}
		From \eqref{eq:qbin_only_poch_sum}, we find
		\begin{align}\label{eq:qbin_ktss_rtss}
			\qbin{k-i+s}{s}=\frac{\pochq{q^{k-i+1}}{s}}{\pochq{q}{s}},\quad \qbin{r-i+s}{s}=\frac{\pochq{q^{r-i+1}}{s}}{\pochq{q}{s}}.
		\end{align}
		Substituting \eqref{eq:qbin_xus} and \eqref{eq:qbin_ktss_rtss} into $L$ gives
		\begin{align*}
			L=\sum_{s=0}^u q^{\binom{u-s}{2}+us-\binom{s}{2}}\frac{\pochq{q}{x}\pochq{q^{-u}}{s}\pochq{q^{k-i+1}}{s}^2}{\pochq{q}{u}\pochq{q}{x-u}\pochq{q^{x-u+1}}{s}\pochq{q^{r-i+1}}{s}\pochq{q}{s}}.
		\end{align*}
		Using \eqref{eq:qbin_only_poch} and rewriting the exponent of the power of $q$ gives
		\begin{align*}
			L=q^{\binom{u}{2}}\qbin{x}{u}\sum_{s=0}^u q^s\frac{\pochq{q^{-u}}{s}\pochq{q^{k-i+1}}{s}^2}{\pochq{q^{x-u+1}}{s}\pochq{q^{r-i+1}}{s}\pochq{q}{s}}.
		\end{align*}
		This becomes
		\begin{align*}
			L=q^{\binom{u}{2}}\qbin{n-r+1}{u}\hyp{q^{-u}, q^{k-i+1}, q^{k-i+1}}{q^{n-r-u+2}, q^{r-i+1}}{q;q}.
		\end{align*}
		By setting $n=u$, $a=b=q^{k-i+1}$, $c=q^{r-i+1}$, $d=q^{n-r-u+2}$ in \eqref{eq:transf32}, we obtain
		\begin{align*}
			L=q^{\binom{u}{2}+u(2k-i-r+1)}\qbin{n-r+1}{u}
			\frac{\pochq{q^{n-2k+i-u+1}}{u}}{\pochq{q^{n-r-u+2}}{u}} \hyp{q^{-u}, q^{r-k}, q^{r-k}}{q^{r-i+1}, q^{n-2k+i-u+1}}{q;q}.
		\end{align*}
		From \eqref{eq:qbin_only_poch_sum}, we find
		\begin{align*}
			L=q^{\binom{u}{2}+u(2k-i-r+1)}\qbin{n-2k+i}{u}
			\hyp{q^{-u}, q^{r-k}, q^{r-k}}{q^{r-i+1}, q^{n-2k+i-u+1}}{q;q}.
		\end{align*}
		We thus have
		\begin{align*}
			L=q^{\binom{u}{2}+u(2k-i-r+1)}\qbin{n-2k+i}{u}
			\sum_{s\geq 0} q^s \frac{\pochq{q^{-u}}{s}\pochq{q^{r-k}}{s}^2}{\pochq{q^{r-i+1}}{s}\pochq{q^{n-2k+i-u+1}}{s}\pochq{q}{s}}.
		\end{align*}
		Applying \eqref{eq:qbin_only_poch_sum} and \eqref{eq:qbin_poch} implies
		\begin{align*}
			L=q^{\binom{u}{2}+u(2k-i-r+1)}\qbin{n-2k+i}{u}
			\sum_{s\geq 0}
			(-1)^s q^{3\binom{s}{2}-us-2s(k-r)+s}
			\frac{\qbin{u}{s}\qbin{k-r}{s}^2}{\qbin{r-i+s}{s}\qbin{n-2k+i-u+s}{s}}.
		\end{align*}
		Since $\qbin{n-2k+i}{u}\qbin{u}{s}=\qbin{n-2k+i}{s}\qbin{n-2k+i-s}{u-s}=\qbin{n-2k+i}{u-s}\qbin{n-2k+i-u+s}{s}$, we obtain
		\begin{align*}
			L=q^{\binom{u}{2}+u(2k-i-r+1)}
			\sum_{s\geq 0}
			(-1)^s q^{3\binom{s}{2}-us-2s(k-r)+s}
			\frac{\qbin{n-2k+i}{u-s}\qbin{k-r}{s}^2}{\qbin{r-i+s}{s}}.
		\end{align*}
		Some standard manipulations give the required identity.
	\end{proof}
\end{lemma}
\begin{lemma}\label{lem:identity2}
	For integers $n,r,k,i$ with $r\leq n+1$, we have
	\begin{multline*}
		\qbin{n-r+1}{i}\sum_{s=0}^i (-1)^s q^{\binom{i-s}{2}}
		\frac{\qbin{r}{i-s}\qbin{k-i+s}{s}^2}{\qbin{n-r-i+s+1}{s}}
		=\\
		q^{i(2k-r-i+1)}\qbin{n-2k+i}{i}\sum_{s=0}^i (-1)^s q^{\binom{i-s}{2}+s(2r-2k+s-1)}
		\frac{\qbin{r}{i-s}\qbin{k-r}{s}^2}{\qbin{n-2k+s}{s}}.
	\end{multline*}
	\begin{proof}
		The required identity follows by setting $u=i$ in Lemma \ref{lem:identity1} and using that for all $s$ with $0\leq s\leq i$, we have 
		\begin{equation*}
			\frac{\qbin{n-r+1}{i-s}\qbin{r}{i}}{\qbin{r-i+s}{s}}=\frac{\qbin{n-r+1}{i}\qbin{r}{i-s}}{\qbin{n-r-i+s+1}{s}}\quad\text{and}\quad\frac{\qbin{n-2k+i}{i-s}\qbin{r}{i}}{\qbin{r-i+s}{s}}=\frac{\qbin{n-2k+i}{i}\qbin{r}{i-s}}{\qbin{n-2k+s}{s}}.\qedhere
		\end{equation*}
	\end{proof}
\end{lemma}
\begin{lemma}\label{lem:identity3}
	For integers $n,r,k,t$ with $t\geq 0$ and $r\leq n+1$, we have
	\begin{align*}
		\sum_{i=0}^t \sum_{s=0}^i (-1)^s q^{i(2k-t-r+1)+\binom{s}{2}+s(2r-2k+s-i)}\frac{\qbin{n-2k+i}{i}\qbin{k-i}{t-i}\qbin{r}{i-s}\qbin{k-r}{s}^2}{\qbin{k}{i}\qbin{n-2k+s}{s}}
		=\frac{\qbin{r}{t}\qbin{n-r+1}{t}}{\qbin{k}{t}}.
	\end{align*}
	\begin{proof}
		Let $L$ denote the left-hand side of the required identity. We have
		\begin{align*}
			L=\sum_{i=0}^t q^{i(2k-t-r+1)-\binom{i}{2}}\frac{\qbin{n-2k+i}{i}\qbin{k-i}{t-i}}{\qbin{k}{i}}
			\sum_{s=0}^i (-1)^s q^{\binom{i-s}{2}+s(2r-2k+s-1)}\frac{\qbin{r}{i-s}\qbin{k-r}{s}^2}{\qbin{n-2k+s}{s}}.
		\end{align*}
		Applying Lemma~\ref{lem:identity2} gives
		\begin{align*}
			L
			&=\sum_{i=0}^t \sum_{s=0}^i (-1)^s q^{i(i-t)-\binom{i}{2}+\binom{i-s}{2}}
			\frac{\qbin{k-i}{t-i}\qbin{n-r+1}{i}\qbin{r}{i-s}\qbin{k-i+s}{s}^2}{\qbin{k}{i}\qbin{n-r-i+s+1}{s}}.
		\end{align*}
		Set $j=i-s$ to obtain
		\begin{align*}
			L
			&=\sum_{i=0}^t \sum_{j=0}^i (-1)^{i-j} q^{i(i-t)-\binom{i}{2}+\binom{j}{2}}
			\frac{\qbin{k-i}{t-i}\qbin{n-r+1}{i}\qbin{r}{j}\qbin{k-j}{i-j}^2}{\qbin{k}{i}\qbin{n-r-j+1}{i-j}}.
		\end{align*}
		We have
		\begin{align*}
			\frac{\qbin{n-r+1}{i}\qbin{k-j}{i-j}^2}{\qbin{k}{i}\qbin{n-r-j+1}{i-j}}=\frac{\qbin{n-r+1}{j}\qbin{k-j}{i-j}}{\qbin{k}{j}}.
		\end{align*}
		We thus obtain
		\begin{align*}
			L
			&=\sum_{i=0}^t \sum_{j=0}^i (-1)^{i-j} q^{i(i-t)-\binom{i}{2}+\binom{j}{2}}
			\frac{\qbin{r}{j}\qbin{n-r+1}{j}\qbin{k-j}{i-j}\qbin{k-i}{t-i}}{\qbin{k}{j}}.
		\end{align*}
		By setting $m=i-j$, we have
		\begin{align*}
			L
			&=\sum_{j=0}^t\sum_{m=0}^{t-j}(-1)^m q^{(m+j)(m+j-t)-\binom{m+j}{2}+\binom{j}{2}}
			\frac{\qbin{r}{j}\qbin{n-r+1}{j}}{\qbin{k}{j}}\qbin{k-j}{m}\qbin{k-j-m}{t-j-m}.
		\end{align*}
		By using $\qbin{k-j}{m}\qbin{k-j-m}{t-j-m}=\qbin{k-j}{k-t}\qbin{t-j}{m}$ and rewriting the exponent in the power of $q$, we obtain
		\begin{align*}
			L
			&=\sum_{j=0}^t q^{j(j-t)}\frac{\qbin{r}{j}\qbin{n-r+1}{j}\qbin{k-j}{k-t}}{\qbin{k}{j}}
			\sum_{m=0}^{t-j}(-1)^m q^{\binom{m}{2}+m(j-t+1)}\qbin{t-j}{m}.
		\end{align*}
		Applying \eqref{eq:qbinomialthm} gives
		\[
		\sum_{m=0}^{t-j}(-1)^m q^{\binom{m}{2}+m(j-t+1)}\qbin{t-j}{m}
		=\prod_{m=0}^{t-j-1}(1-q^{j-t+1+m})
		=\begin{cases}
			0&\text{if $t-j-1\geq 0$}\\
			1&\text{if $t=j$}.
		\end{cases}
		\]
		Therefore, we obtain
		\begin{align*}
			L=\frac{\qbin{r}{t}\qbin{n-r+1}{t}}{\qbin{k}{t}},
		\end{align*}
		as required.
	\end{proof}
\end{lemma}
\begin{proposition}\label{prop:threesums1}
	For integers $n,k,r,t$ with $t\geq 0$ and $r\leq n+1$, we have
	\begin{multline*}
		\sum_{i=0}^t \sum_{j=0}^{t-i} \sum_{s=0}^i (-1)^{i+j+s}q^{-(k-i)^2+\binom{j}{2}+\binom{s+r-i}{2}+(k-s-r)(k-s)}
		\frac{\qbin{n-2k+i}{i}\qbin{k-i}{j}\qbin{n-i-j}{t-i-j}\qbin{r}{i-s}\qbin{k-r}{s}^2}{\qbin{k}{i}\qbin{k-i-j}{t-i-j}\qbin{n-2k+s}{s}}\\
		=(-1)^t q^{\binom{r}{2}-kr+\binom{t+1}{2}}\frac{\qbin{r-1}{t}\qbin{n-r}{t}}{\qbin{k}{t}}.
	\end{multline*}
	\begin{proof}
		We denote the left-hand side of the required identity by $L(t)$ and its summand by $a(t,i,j,s)$, so that we have
		\[
		L(t)=\sum_{i=0}^t \sum_{j=0}^{t-i} \sum_{s=0}^i a(t,i,j,s).
		\]
		We then obtain
		\begin{align*}
			L(t+1)
			&=\sum_{i=0}^{t+1} \sum_{j=0}^{t+1-i} \sum_{s=0}^i a(t+1,i,j,s)\\
			&=\sum_{i=0}^t \sum_{j=0}^{t-i} \sum_{s=0}^i a(t+1,i,j,s)+\sum_{i=0}^{t+1} \sum_{s=0}^i a(t+1,i,t+1-i,s).
		\end{align*}
		Using $\qbin{n}{k}=\frac{\qnum{n}!}{\qnum{k}!\qnum{n-k}!}$ gives
		\[
		a(t+1,i,j,s)
		=a(t,i,j,s)\frac{\qbin{n-i-j}{t+1-i-j}\qbin{k-i-j}{t-i-j}}{\qbin{k-i-j}{t+1-i-j}\qbin{n-i-j}{t-i-j}}
		=a(t,i,j,s)\frac{\qnum{n-t}}{\qnum{k-t}}.
		\]
		Therefore, we have
		\begin{align*}
			L(t+1)
			&=\frac{\qnum{n-t}}{\qnum{k-t}}L(t)+\sum_{i=0}^{t+1} \sum_{s=0}^i a(t+1,i,t+1-i,s).
		\end{align*}
		It holds that
		\begin{align*}
			&\sum_{i=0}^{t+1} \sum_{s=0}^i a(t+1,i,t+1-i,s)\\
			=&(-1)^{t+1}\sum_{i=0}^{t+1} \sum_{s=0}^i (-1)^s q^{-(k-i)^2+\binom{t+1-i}{2}+\binom{s+r-i}{2}+(k-s-r)(k-s)}
			\frac{\qbin{n-2k+i}{i}\qbin{k-i}{t+1-i}\qbin{r}{i-s}\qbin{k-r}{s}^2}{\qbin{k}{i}\qbin{n-2k+s}{s}}\\
			=&(-1)^{t+1}q^{\binom{r}{2}-rk+\binom{t+1}{2}}\sum_{i=0}^{t+1} \sum_{s=0}^i (-1)^s q^{i(2k-t-r)+\binom{s}{2}+s(2r-2k+s-i)}
			\frac{\qbin{n-2k+i}{i}\qbin{k-i}{t+1-i}\qbin{r}{i-s}\qbin{k-r}{s}^2}{\qbin{k}{i}\qbin{n-2k+s}{s}}.
		\end{align*}
		By Lemma~\ref{lem:identity3}, we then obtain 
		\begin{align*}
			\sum_{i=0}^{t+1} \sum_{s=0}^i a(t+1,i,t+1-i,s)
			=(-1)^{t+1}q^{\binom{r}{2}-rk+\binom{t+1}{2}}\frac{\qbin{r}{t+1}\qbin{n-r+1}{t+1}}{\qbin{k}{t+1}}.
		\end{align*}
		Hence, we have
		\begin{align}\label{eq:expressionft1}
			L(t+1)
			&=\frac{\qnum{n-t}}{\qnum{k-t}}L(t)+(-1)^{t+1}q^{\binom{r}{2}-rk+\binom{t+1}{2}}\frac{\qbin{r}{t+1}\qbin{n-r+1}{t+1}}{\qbin{k}{t+1}}.
		\end{align}
		The proposition is now proven by induction on $t$. The required identity holds for $t=0$ since $\qbin{n}{0}=1$. Assume that it holds for a given $t>0$. By using \eqref{eq:expressionft1}, we obtain
		\begin{align*}
			L(t+1)
			&=\frac{\qnum{n-t}}{\qnum{k-t}}(-1)^t q^{\binom{r}{2}-kr+\binom{t+1}{2}}\frac{\qbin{r-1}{t}\qbin{n-r}{t}}{\qbin{k}{t}}+(-1)^{t+1}q^{\binom{r}{2}-rk+\binom{t+1}{2}}\frac{\qbin{r}{t+1}\qbin{n-r+1}{t+1}}{\qbin{k}{t+1}}\\
			&=(-1)^{t+1}q^{\binom{r}{2}-kr+\binom{t+1}{2}}
			\frac{\qbin{r-1}{t+1}\qbin{n-r}{t+1}}{\qbin{k}{t+1}}\left(
			-\frac{\qnum{n-t}\qnum{t+1}}{\qnum{r-t-1}\qnum{n-r-t}}+\frac{\qnum{r}\qnum{n-r+1}}{\qnum{r-t-1}\qnum{n-r-t}}\right)\\
			&=(-1)^{t+1}q^{\binom{r}{2}-kr+\binom{t+2}{2}}
			\frac{\qbin{r-1}{t+1}\qbin{n-r}{t+1}}{\qbin{k}{t+1}}.
		\end{align*} 
		This completes the proof.
	\end{proof}
\end{proposition}
\begin{proposition}\label{prop:threesums2}
	Let $n,k,r,t$ be integers with $r\leq n+1$. Then we have
	\begin{multline*}
		\sum_{a=0}^t (-1)^a c_a\frac{\qbin{r-1}{a}\qbin{n-r}{a}}{\qbin{k}{a}}\\
		=\sum_{i=0}^t \sum_{j=0}^{t-i} \sum_{s=0}^i 
		(-1)^{i+j+s}q^{-\binom{r}{2}+kr-(k-i)^2+\binom{j}{2}+\binom{s+r-i}{2}+(k-s-r)(k-s)}
		\frac{\qbin{n-2k+i}{i}\qbin{k-i}{j}\qbin{r}{i-s}\qbin{k-r}{s}^2}{\qbin{k}{i}\qbin{n-2k+s}{s}}.
	\end{multline*}
	where
	\[
	c_a=\begin{cases}
		q^{\binom{t+1}{2}}&\text{for $a=t$,}\\
		q^{\binom{a+1}{2}+n-a}\frac{\qnum{k-n}}{\qnum{k-a}}&\text{for $0\leq a<t$.}
	\end{cases}
	\]
	\begin{proof}
		Proposition~\ref{prop:threesums1} gives
		\begin{multline*}
			\sum_{a=0}^t (-1)^a c_a\frac{\qbin{r-1}{a}\qbin{n-r}{a}}{\qbin{k}{a}}\\
			=\sum_{a=0}^t\sum_{i=0}^a \sum_{j=0}^{a-i} \sum_{s=0}^i c_a (-1)^{i+j+s}q^{-\binom{r}{2}+kr-\binom{a+1}{2}-(k-i)^2+\binom{j}{2}+\binom{s+r-i}{2}+(k-s-r)(k-s)}\\
			\times
			\frac{\qbin{n-2k+i}{i}\qbin{k-i}{j}\qbin{n-i-j}{a-i-j}\qbin{r}{i-s}\qbin{k-r}{s}^2}{\qbin{k}{i}\qbin{k-i-j}{a-i-j}\qbin{n-2k+s}{s}}.
		\end{multline*}
		This becomes
		\begin{multline*}
			\sum_{a=0}^t (-1)^a c_a\frac{\qbin{r-1}{a}\qbin{n-r}{a}}{\qbin{k}{a}}
			=\sum_{i=0}^t \sum_{j=0}^{t-i} \sum_{s=0}^i \Bigg(\sum_{a=i+j}^t c_a q^{-\binom{a+1}{2}} \frac{\qbin{n-i-j}{a-i-j}}{\qbin{k-i-j}{a-i-j}} \Bigg)\\
			\times (-1)^{i+j+s}q^{-\binom{r}{2}+kr-(k-i)^2+\binom{j}{2}+\binom{s+r-i}{2}+(k-s-r)(k-s)}
			\frac{\qbin{n-2k+i}{i}\qbin{k-i}{j}\qbin{r}{i-s}\qbin{k-r}{s}^2}{\qbin{k}{i}\qbin{n-2k+s}{s}}.
		\end{multline*}
		We now prove by induction that, for all $m=0,1,\dots,t$, we have
		\[
		\sum_{a=m}^t c_a q^{-\binom{a+1}{2}} \frac{\qbin{n-m}{a-m}}{\qbin{k-m}{a-m}}=1.
		\]
		Observe that this equality holds for $m=t$. Assume that equality holds for an $m\in\{1,\dots,t\}$. We then obtain
		\begin{align*}
			\sum_{a=m-1}^t c_a q^{-\binom{a+1}{2}} \frac{\qbin{n-m+1}{a-m+1}}{\qbin{k-m+1}{a-m+1}}
			&=c_{m-1}q^{-\binom{m}{2}}+\sum_{a=m}^t c_a q^{-\binom{a+1}{2}} \frac{\qnum{n-m+1}}{\qnum{k-m+1}}\frac{\qbin{n-m}{a-m}}{\qbin{k-m}{a-m}}\\
			&=q^{n-m+1}\frac{\qnum{k-n}}{\qnum{k-m+1}}+\frac{\qnum{n-m+1}}{\qnum{k-m+1}}=1,
		\end{align*}
		which proves the required identity.
	\end{proof}
\end{proposition}
We can now derive the eigenvalues of the matrix $UU^{T}$.
\begin{theorem}\label{thm_eigenvalue.simplified}
	The eigenvalues of $UU^{T}$ are given by
	\begin{align*}
		\mu_{0}&=\kappa\left(1-\frac{\qnum{k-n}}{\qnum{k}} \sum_{i=0}^{t-1} q^{n-i}\frac{\qbin{n}{i}}{\qbin{k-1}{i}}\right)\\
		\mu_{r}&=\kappa\left(1+(-1)^r q^{\binom{r}{2}-kr+k}\frac{1}{\qbin{n-k-1}{r-1}} \sum_{i=0}^{t-1} (-1)^i q^{\binom{i}{2}}\qbin{k-i-1}{r-i-1}\qbin{n-r}{i}\right)
	\end{align*}
	for $r=1,2,\dots,k$.
	\begin{proof}
		The eigenvalues $\mu_{r}$ for $0\leq r\leq k$ are given by
		\begin{align}\label{eq:mur}
			\mu_{r}=\kappa+\sum_{i=0}^t \kappa_i \nu_{r}^{(k-i)}.
		\end{align}
		By using from Proposition \ref{prop_kappa} that
		\begin{align}\label{eq:kappa}
			\kappa = N\frac{\qbin{n}{t}}{\qbin{n}{k}\qbin{k}{t}}=\frac{N}{\qbin{n-t}{k-t}}
		\end{align}
		together with Proposition \ref{prop_lv.wang.binomial} and \eqref{eq:uppernegation}, the value of $\kappa_{i}$ from Proposition \ref{prop_kappa.i} can be expressed as
		\begin{align*}
			\kappa_{i}&=\frac{\kappa }{\qbin{n-k}{k-i}q^{(k-i)^2}}
			\Bigg(\frac{1}{\qbin{n-t}{k-t}}
			\sum_{j=0}^{t-i}(-1)^j q^{\binom{j}{2}} \qbin{k-i}{j}\qbin{n-i-j}{k-i-j}
			-\sum_{j=0}^{t-i}(-1)^j q^{\binom{j}{2}} \qbin{k-i}{j}\Bigg).
		\end{align*}
		Since it holds that
		\[
		\qbin{n-k}{k-i}=\frac{\qbin{k}{i}\qbin{n-k}{k}}{\qbin{n-2k+i}{i}}\quad\text{and}\quad \frac{\qbin{n-i-j}{k-i-j}}{\qbin{n-t}{k-t}}=\frac{\qbin{n-i-j}{t-i-j}}{\qbin{k-i-j}{t-i-j}},
		\]
		we obtain
		\begin{align}\label{eq:kappai}
			\kappa_{i}&=\frac{\kappa \qbin{n-2k+i}{i}}{\qbin{k}{i}\qbin{n-k}{k}q^{(k-i)^2}}
			\Bigg(\sum_{j=0}^{t-i}(-1)^j q^{\binom{j}{2}} \frac{\qbin{k-i}{j}\qbin{n-i-j}{t-i-j}}{\qbin{k-i-j}{t-i-j}}
			-\sum_{j=0}^{t-i}(-1)^j q^{\binom{j}{2}} \qbin{k-i}{j}\Bigg).
		\end{align}
		By \eqref{eq_eigenvalue.grassmann.scheme.delsarte}, we have
		\begin{align*}
			\nu_{r}^{(k-i)}=\sum_{j=0}^{k-i} (-1)^j \qbin{r}{j}\qbin{k-r}{i+j-r}\qbin{n-k-r}{k-i-j} q^{\binom{j}{2}+(k-i-j)(k-i-j+r)}.
		\end{align*}
		Set $s=i+j-r$ to obtain
		\begin{align*}
			\nu_{r}^{(k-i)}=(-1)^r\sum_{s=0}^{i} (-1)^{s+i} \qbin{r}{i-s}\qbin{k-r}{s}\qbin{n-k-r}{k-s-r} q^{\binom{s+r-i}{2}+(k-s-r)(k-s)}.
		\end{align*}
		Using
		\[
		\qbin{n-k-r}{k-s-r}=\frac{\qbin{n-k-r}{k-r}\qbin{k-r}{s}}{\qbin{n-2k+s}{s}}
		\]
		gives
		\begin{align}\label{eq:eigvalGrassmann}
			\nu_{r}^{(k-i)}=(-1)^r\qbin{n-k-r}{k-r}\sum_{s=0}^{i} (-1)^{s+i} \frac{\qbin{r}{i-s}\qbin{k-r}{s}^2}{\qbin{n-2k+s}{s}} q^{\binom{s+r-i}{2}+(k-s-r)(k-s)}.
		\end{align}
		Substituting \eqref{eq:kappa}, \eqref{eq:kappai}, and \eqref{eq:eigvalGrassmann} into \eqref{eq:mur} gives
		\begin{multline*}
			\mu_{r}=\kappa\Bigg(1+(-1)^r \frac{\qbin{n-k-r}{k-r}}{\qbin{n-k}{k}}\sum_{i=0}^t \sum_{s=0}^{i} (-1)^{s+i}q^{-(k-i)^2+\binom{s+r-i}{2}+(k-s-r)(k-s)} \frac{\qbin{n-2k+i}{i}\qbin{r}{i-s}\qbin{k-r}{s}^2}{\qbin{k}{i}\qbin{n-2k+s}{s}} \\
			\times
			\Bigg(\sum_{j=0}^{t-i}(-1)^j q^{\binom{j}{2}} \frac{\qbin{k-i}{j}\qbin{n-i-j}{t-i-j}}{\qbin{k-i-j}{t-i-j}}
			-\sum_{j=0}^{t-i}(-1)^j q^{\binom{j}{2}} \qbin{k-i}{j}\Bigg)\Bigg).
		\end{multline*}
		We write
		\begin{align}\label{eq:mur_short}
			\mu_{r}=\kappa\left(1+(-1)^r\frac{\qbin{n-k-r}{k-r}}{\qbin{n-k}{k}}(A-B)\right),
		\end{align}
		where
		\begin{align*}
			A=\sum_{i=0}^t \sum_{j=0}^{t-i} \sum_{s=0}^{i} (-1)^{i+j+s}q^{-(k-i)^2+\binom{j}{2}+\binom{s+r-i}{2}+(k-s-r)(k-s)} \frac{\qbin{n-2k+i}{i}\qbin{r}{i-s}\qbin{k-r}{s}^2\qbin{k-i}{j}\qbin{n-i-j}{t-i-j}}{\qbin{k}{i}\qbin{n-2k+s}{s}\qbin{k-i-j}{t-i-j}}
		\end{align*}
		and
		\begin{align*}
			B=\sum_{i=0}^t \sum_{j=0}^{t-i} \sum_{s=0}^{i} (-1)^{i+j+s}q^{-(k-i)^2+\binom{j}{2}+\binom{s+r-i}{2}+(k-s-r)(k-s)} \frac{\qbin{n-2k+i}{i}\qbin{r}{i-s}\qbin{k-r}{s}^2\qbin{k-i}{j}}{\qbin{k}{i}\qbin{n-2k+s}{s}}.
		\end{align*}
		From Proposition~\ref{prop:threesums1}, we find that
		\begin{align*}
			A=(-1)^t q^{\binom{r}{2}-kr+\binom{t+1}{2}}\frac{\qbin{r-1}{t}\qbin{n-r}{t}}{\qbin{k}{t}},
		\end{align*}
		and from Proposition~\ref{prop:threesums2} that
		\begin{align*}
			B=q^{\binom{r}{2}-kr} \sum_{a=0}^t (-1)^a c_a \frac{\qbin{r-1}{a}\qbin{n-r}{a}}{\qbin{k}{a}}.
		\end{align*}
		This implies
		\begin{align}\label{eq:AdiffB}
			A-B=-q^{\binom{r}{2}-kr} \sum_{a=0}^{t-1}(-1)^a q^{\binom{a+1}{2}+n-a}\frac{\qnum{k-n}}{\qnum{k-a}} \frac{\qbin{r-1}{a}\qbin{n-r}{a}}{\qbin{k}{a}}.
		\end{align}
		By substituting \eqref{eq:AdiffB} into \eqref{eq:mur_short}, we obtain
		\begin{align*}
			\mu_{r}=\kappa\left(1-(-1)^r\frac{\qbin{n-k-r}{k-r}}{\qbin{n-k}{k}}q^{\binom{r}{2}-kr} \sum_{a=0}^{t-1} (-1)^a q^{\binom{a+1}{2}+n-a}\frac{\qnum{k-n}}{\qnum{k-a}} \frac{\qbin{r-1}{a}\qbin{n-r}{a}}{\qbin{k}{a}}\right).
		\end{align*}
		Since $\qbin{-1}{a}=(-1)^a q^{-\binom{a+1}{2}}$ and $\qnum{k-a}\qbin{k}{a}=\qnum{k}\qbin{k-1}{a}$, we have
		\begin{align*}
			\mu_{0}
			=\kappa\left(1-\sum_{a=0}^{t-1} (-1)^a q^{\binom{a+1}{2}+n-a}\frac{\qnum{k-n}}{\qnum{k-a}} \frac{\qbin{-1}{a}\qbin{n}{a}}{\qbin{k}{a}}\right)
			=\kappa\left(1-\frac{\qnum{k-n}}{\qnum{k}} \sum_{a=0}^{t-1} q^{n-a}\frac{\qbin{n}{a}}{\qbin{k-1}{a}}\right),
		\end{align*}
		as required.
		For $r\geq 1$, we use
		\[
		\frac{\qbin{r-1}{a}}{\qnum{k-a}\qbin{k}{a}}
		=\frac{\qbin{r-1}{a}}{\qnum{k}\qbin{k-1}{a}}
		=\frac{\qbin{k-a-1}{r-a-1}}{\qnum{k}\qbin{k-1}{r-1}}
		\]
		to obtain 
		\begin{align*}
			\mu_{r}=\kappa\left(1-(-1)^r\frac{\qnum{k-n}\qbin{n-k-r}{k-r}}{\qnum{k}\qbin{k-1}{r-1}\qbin{n-k}{k}}q^{\binom{r}{2}-kr} \sum_{a=0}^{t-1} (-1)^a q^{\binom{a+1}{2}+n-a}\qbin{k-a-1}{r-a-1}\qbin{n-r}{a}\right).
		\end{align*}
		Since
		\[
		\frac{\qnum{k-n}\qbin{n-k-r}{k-r}}{\qnum{k}\qbin{k-1}{r-1}\qbin{n-k}{k}}=-q^{k-n}\frac{1}{\qbin{n-k-1}{r-1}},
		\]
		we have
		\begin{align*}
			\mu_{r}=\kappa\left(1+(-1)^r q^{\binom{r}{2}-kr+k}\frac{1}{\qbin{n-k-1}{r-1}} \sum_{a=0}^{t-1} (-1)^a q^{\binom{a}{2}}\qbin{k-a-1}{r-a-1}\qbin{n-r}{a}\right).
		\end{align*}
		This completes the proof.
	\end{proof}
\end{theorem}

\section{Proof of Theorem \ref{thm_main}}
\begin{proof}[Proof of Theorem \ref{thm_main}]
	Recall that the dimension $\dim_{\mathbb{Q}}\mathfrak{V}$ equals ${n\brack k}_{q}$ minus the multiplicity of zero as an eigenvalue of the matrix $UU^{T}$. It is straightforward to see that $\mu_{0}\neq 0$. We have
	\begin{align}
		\sum_{i=0}^{t-1}(-1)^{i}q^{\binom{i}{2}}{k-i-1\brack r-i-1}{n-r\brack i}
		&=q^{(r-1)(n-r)}{(k-1)-(n-r)\brack r-1}\notag\\
		&=(-1)^{r-1}q^{kr-k-\binom{r}{2}}{n-k-1\brack r-1},\label{eq:identity_used_for_mu_r}
	\end{align}
	where the first equality follows from Proposition \ref{prop_kurihara.triterms} by setting $(x,y,p,h)=(n-r,k-1,r-1,0)$ and the second equality follows from the upper negation formula \eqref{eq:uppernegation}. By Theorem \ref{thm_eigenvalue.simplified}, this implies $\mu_{r}=0$ for all $1\leq r\leq t$.\par 
	To show that $\mu_{r}$ is nonzero when $t+1\leq r\leq n$, we need to introduce some terminologies. For each integer $m$, we consider the $q$-adic expression of $m$, that is to write $m$ in the form of $\pm\sum_{i\geq 0}c_{i}q^{i}\in\mathbb{Z}[q]$, where $0\leq c_{i}<q$. Then we define \textit{the smallest exponent of $q$ in $m$} to be the natural number
	\begin{align*}
		v_{q}(m):=\begin{cases}
			\min\left\{i\geq 0\mid c_{i}\neq 0\right\} & \text{if $m\neq 0$},\\
			\infty & \text{if $m=0$},
		\end{cases}
	\end{align*}
	which is also known as the \textit{$q$-adic valuation of $m$}. We note that all $q$-binomial coefficients are of the form $\sum_{i\geq 0}c_{i}q^{i}$ with $c_{0}=1$ by definition. Having this notation, for $t+1\leq r\leq k$ we compute from \eqref{eq:identity_used_for_mu_r}
	\begin{align*}
		\sum_{i=0}^{t-1}(-1)^{i}q^{\binom{i}{2}}{k-i-1\brack r-i-1}{n-r\brack i}=(-1)^{r-1}q^{kr-k-\binom{r}{2}}{n-k-1\brack r-1}-\sum_{i=t}^{r-1}(-1)^{i}q^{\binom{i}{2}}{k-i-1\brack r-i-1}{n-r\brack i}.
	\end{align*}
	Note that by $r\leq k$ we have $\binom{r-1}{2}<kr-k-\binom{r}{2}$. Hence $v_{q}\left(\sum_{i=0}^{t-1}(-1)^{i}q^{\binom{i}{2}}{k-i-1\brack r-i-1}{n-r\brack i}\right)=\binom{t}{2}$, which by Theorem \ref{thm_eigenvalue.simplified} indicates that $\mu_{r}\neq 0$ for all $t+1\leq r\leq k$. Therefore the multiplicity of zero as an eigenvalue of $UU^{T}$ equals
	\begin{align*}
		\sum_{r=1}^{t}\left({n\brack r}-{n\brack r-1}\right)={n\brack t}-1,
	\end{align*}
	which completes the proof.
\end{proof}

\section*{Acknowledgement}
Charlene Wei{\ss} was supported by the German Academic Exchange Service (DAAD) for a research stay at the University of Amsterdam and thanks the Korteweg de Vries Institute for Mathematics for its hospitality; during this stay, she became involved in this project. Yue Zhou is partially supported by the Natural Science Foundation of Hunan Province (No.~2023RC1003) and the National Natural Science Foundation of China (No.~12371337).


\bibliographystyle{amsalpha}
\bibliography{references}

\end{document}